\pdfoutput=1 
\documentclass[12pt]{amsart}
\usepackage{graphicx}
\usepackage[T1]{fontenc}
\usepackage[unicode]{hyperref}
\usepackage[alphabetic]{amsrefs}
\usepackage{amssymb,amsmath,amscd,amsthm,mathtools}
\makeatletter
\def\@settitle{\begin{center}%
		\baselineskip14\p@\relax
		\bfseries
		\@title
	\end{center}
}
\makeatother

\let\originalleft\left
\let\originalright\right
\renewcommand{\left}{\mathopen{}\mathclose\bgroup\originalleft}
\renewcommand{\right}{\aftergroup\egroup\originalright}

\theoremstyle{plain}
\newtheorem{thm}{Theorem}[section]
\newtheorem{remk}[thm]{Remark}
\newtheorem{prop}[thm]{Proposition}

\newtheorem{lemma}[thm]{Lemma}

\theoremstyle{definition}
\newtheorem{defn}[thm]{Definition}

\theoremstyle{definition}

\theoremstyle{remark}
\newtheorem*{remk*}{Remark}

\def \F {\mathcal{F}}
\def \G {\mathcal{G}}
\def \H {\mathcal{H}}
\def \kpt {\operatorname{pt}}
\def \O {\mathcal{O}}
\def \L {\mathcal{L}}
\def \C {\mathcal{C}}
\def \M {\mathcal{M}}
\def \mainfield {\mathrm{k}}
\def \K {\mathbb K}
\def \KK {\mathbf K}
\makeatletter
\newcommand*{\defeq}{\mathrel{\rlap{%
			\raisebox{0.3ex}{$\m@th\cdot$}}%
		\raisebox{-0.3ex}{$\m@th\cdot$}}%
	=}
\makeatother
\newcommand{\relmiddle}[1]{\mathrel{}\middle#1\mathrel{}}
\newcommand{\RN}[1]{
	\textup{\uppercase\expandafter{\romannumeral#1}}
}
\date{}
\title{Second differentials in the Quillen spectral sequence}
\author{Georgy Belousov}
\email{jorabelousov@gmail.com}

\begin{document}
\begin{abstract}
	For an algebraic variety $X$, we introduce generalized first Chern classes, which are defined for coherent sheaves on $X$ with support in codimension $p$ and take values in $CH^{p+1}(X)$. We use them to provide an explicit formula for the differentials $d_2^p: E_2^{p,-p-1} \to E_2^{p+2, -p-2} \cong CH^{p+2}(X)$ in the Quillen spectral sequence. If the variety $X$ is regular, this gives maps $H^p(X, \mathcal{K}_{p+1}) \to CH^{p+2}(X)$.
\end{abstract}
\maketitle
\begin{section}{Introduction}
Sheaves of $K$-groups naturally arise in the study of algebraic cycles.
The first phenomenon of this sort was Bloch's formula, identifying $p$-th cohomology of the sheaf of $p$-th algebraic $K$-groups on a smooth variety $X$ with the $p$-th Chow group of the variety.
This was proved in the case $p=2$ by Bloch (see \cite{Bloch}), and later verified for all $p$ by Quillen (see \cite{Quillen}.) 

More generally, there is a spectral sequence starting with $K$-cohomology, that is, cohomology of sheaves of $K$-groups on $X$, and converging to $K$-groups of $X$ (see \cite{BrownGersten}).
By Gersten's conjecture, proved by Quillen, for a smooth variety this spectral sequence coincides with the Quillen spectral sequence (we recall the definition in Section \ref{statement}).
Its terms are cohomology of Gersten complexes, and are given by certain groups of equivalence classes of algebraic cycles equipped with elements of $K$-groups of function fields of their components.

Take, for example, the first nontrivial differential on the second sheet of this spectral sequence, namely 
\begin{equation*}
	{d_2: H^1(X, \mathcal{K}_2)\rightarrow H^3(X, \mathcal{K}_3)} \cong CH^3(X).
\end{equation*}
Considering the Gersten complex, one can see that an element of the left hand side is identified with a collection of divisors on $X$, each of which is endowed with a non-zero rational function, subject to a certain cancellation condition on their divisors.
The target group of $d_2$ coincides with $CH^3(X)$ by Bloch's formula.

Consequently, both the source and the target of $d_2$ are expressed in terms of algebraic cycles and rational functions on them.
This raises the question of whether the operation $d_2$ admits a geometric description. Such a description is the topic of this work.

\medskip \noindent With this aim, we introduce the maps $c_1^p$, generalizing the first Chern class.
Let $\F$ be a coherent sheaf on $X$ \emph{with support in codimension} $p$.
In this case we can construct an element $c_1^p(\F)\in CH^{p+1}(X)$ using the Quillen spectral sequence (see Section \ref{Chern}). We also give a more explicit formula for $c_1^p(\F)$ in terms of the first Chern class on smooth varieties (see Proposition \ref{chern-explicit}).
For instance, if $Z$ is a smooth closed subvariety of codimension $p$ in $X$, $\G$ is a coherent sheaf on $Z$, and $j:Z \hookrightarrow X$ denotes the corresponding closed immersion map, then \begin{equation*}c_1^p(j_*(\G)) = j_* (c_1(\G)) \in CH^{p+1}(X),\end{equation*}
where $c_1(\G) \in CH^1(Z)$.

We prove the results for arbitrary equidimensional, not necessarily smooth algebraic varieties. 
Without the smoothness assumption, the terms of the Quillen spectral sequence $E_r^{p,q}$ are no longer given by $K$-cohomology. 
However, the expression of $E_r^{p,q}$ in terms of algebraic cycles remains valid.

Specifically, an element $\alpha\in E_2^{p,-p-1}$ is presented by a finite collection $\{(W_i,\varphi_i)\}_i$ of codimension $p$ cycles $W_i$ on $X$, each of which is endowed with a non-zero rational function $\varphi_i$, such that the sum of the divisors of $\varphi_i$ is zero as a codimension $(p+1)$ cycle on $X$.
In Section \ref{statement}, we define coherent sheaves $\F_i$ and $\G_i$ on $X$ with support on $W_i$; roughly, if the numerator and the denominator of $\varphi_i$ are considered as sections of a line bundle, then $\F_i$ and $\G_i$ are cokernels of these sections.

Our main result (see Theorem \ref{main}) is that the equality 
\begin{equation}\label{intro_formula} 
	d_2^p(\alpha)=\sum_i\big(c_1^{p+1}(\F_i)-c_1^{p+1}(\G_i)\big)
\end{equation} 
holds in $CH^{p+2}(X)$ for $p \geq 1$, where 
\begin{equation*} 
d_2^p: E_2^{p, -p-1} \to E_2^{p+2, -p-2} \cong CH^{p+2}(X)
\end{equation*}
is a differential in the Quillen spectral sequence. 

Let $X$ be a smooth variety over a field of characteristic $0$. In this case we have the comparison between Chow groups and adjoint quotients of the $K_0$-group (see \cite{Fulton}*{Example $15.3.6$}). Since the target of $d_2^p$ is on the last nonvanishing diagonal of the spectral sequence, it follows that the image of $d_2^p$ lies in the $(p+1)!$-torsion of $CH^{p+2}(X)$. We can also see directly that the right-hand side of formula \eqref{intro_formula} is a $(p+1)!$-torsion class. Indeed, as explained in Remark \ref{RR} below, under our assumptions $(p+1)!\:c_1^{p+1}$ equals (up to sign) the usual $(p+2)$-nd Chern class. In other words, $c_1^{p+1}$ provides a natural way to divide by $(p+1)!$ the $(p+2)$-nd Chern class of a sheaf with support in codimension $p+1$. Now by construction $[\F_i]=[\G_i]$ in $K_0(X)$, so the difference of their Chern classes is zero in $CH^{p+2}(X)$.

\medskip \noindent The orders of differentials in the Atiyah--Hirzebruch spectral sequence in topology were computed by V. Buchstaber through the use of the action of cohomological operations \cite{Buchstaber}.
He was also able to give a formula for certain differentials in terms of Steenrod squares.

A similar idea was later used by A. Merkurjev to estimate the orders of the differentials in the Quillen spectral sequence with targets at the $K_0$- and $K_1$-diagonals (see \cite{Merkurjev}).

S. Yagunov carried out a similar program for the motivic spectral sequence in his recent paper \cite{Yagunov}. 
For a prime $l \geq 3$, he proved that the $l$-local part of $d_r$ vanishes for $r < l$ and obtained an expression for the $l$-local part of $d_l$ in terms of stable motivic operations of Voevodsky. 
A crucial new ingredient in his proof is a computation of the motivic Steenrod algebra with $l$-cyclotomic coefficients.
The motivic spectral sequence is closely related to the Quillen spectral sequence considered in the present article. In particular, there is a differential in the motivic spectral sequence whose source and target are isomorphic to those of $d_2^p$ (see \cite{Landsburg}*{Theorem 2.5}). However, the method of Yagunov gives no information about this differential because of certain trivial vanishings occurring when $l=2$.
We also note that our answer is given in explicit geometric terms, in contrast to the expressions for higher differentials found by S. Yagunov.
\end{section}
\begin{section}{Statement of the main result}\label{statement}
	Let $X$ be an integral scheme of finite type over a field $\mainfield$. For convenience, we assume that $X$ is equidimensional.
	
	The abelian category $\M_X$ of coherent sheaves on $X$ is equipped with a finite-length filtration 
	\begin{equation*}\M^p_X=\left\{\mathcal{F}\in\M_X \relmiddle{|}\operatorname{codim}_X(\operatorname{Supp}(\mathcal{F}))\geq p\right\}
	\end{equation*}
   	by Serre subcategories, called the filtration by codimension of support.
	Taking $K$-groups of this filtered abelian category yields an exact couple
	\begin{equation*} K_i(\M_X^{p+1}) \xrightarrow{\partial} K_i(\M_X^{p}) \to 
	\bigoplus_{x \in X^{(p)}} K_i (\mainfield(x)) \to 
	K_{i-1}(\M_X^{p+1}).\end{equation*} 
	The Quillen spectral sequence is, by definition, the fourth quadrant cohomological spectral sequence \[E_1^{p,q}=\bigoplus_{x\in X^{(p)}} K_{-p-q}(\mainfield(x))
	\Rightarrow K_{-p-q}(X)\] associated with this exact couple (see \cite{Quillen}*{Theorem 5.4}.)
	
	In this paper, we compute the differentials 
	\begin{equation*} d_2^p: E_2^{p, -p-1} \to E_2^{p+2, -p-2} \cong CH^{p+2}(X), \quad p \geq 1.
	\end{equation*}
	
	An element $\alpha\in E_2^{p,-p-1}$ is presented by a finite collection $\{(W_i,\varphi_i)\}_i$ of integral codimension $p$ cycles $W_i$ on $X$, each of which is endowed with a non-zero rational function $\varphi_i$, such that the equality $\sum_i \operatorname{div}(\varphi_i) = 0$ holds in the group $Z^{p+1}(X)$ of codimension $(p+1)$ cycles on $X$ (not taken up to rational equivalence).
	
	For each $i$, denote by $j_i$ the closed immersion $W_i \hookrightarrow X$.
	Let $\varphi_i = \frac{f_i}{g_i}$, where
	\begin{equation*}
	f_i, g_i : \O_{W_i} \to \L_i
	\end{equation*}
	are regular sections of a line bundle $\L_i$. Let $D_i, E_i$ be the effective Cartier divisors of zeros of the sections $f_i, g_i$ (see, for example, \cite{div0}*{Definition 2.1.8}.) By construction, we have \[ \operatorname{div}(\varphi_i) = D_i - E_i. \]
	Define the following coherent sheaves on $X$: \begin{equation*} \F_i=j_{i,*}\,(\operatorname{Coker}(f_i)),\quad
	\G_i=j_{i,*}\,(\operatorname{Coker}(g_i)). \end{equation*}
	
	In Section \ref{Chern}, we introduce the maps \begin{equation*}c^{p}_1:K_0(\M^p_X)\to CH^{p+1}(X)\end{equation*} obtained by applying the construction of the first Chern class to the filtered category $\M_X^p$ in place of $\M_X$.
	Our main result is then the following.
	\begin{thm}\label{main} 
		In $CH^{p+2}(X)$ there is an identity 
		\begin{equation*}d_2^p(\alpha)=\sum_i\big(c_1^{p+1}(\F_i)-c_1^{p+1}(\G_i)\big).\end{equation*}
	\end{thm}
\end{section}
\begin{section}{Generalized first Chern classes}\label{Chern}
	In this section we define the maps $c_0^p$ and $c_1^p$, repeating one of the constructions of the usual zeroth and first Chern class with $\M_X$ replaced by $\M_X^p$.
	In particular, $c_0^0=c_0,\: c_1^0=c_1$, where $c_0$ and $c_1$ are considered on arbitrary, not necessarily smooth varieties.
			
	Consider the first sheet of the Quillen spectral sequence for the category $\M_X^p$ with the filtration by codimension of support (i.e., the stupid truncation of the Quillen spectral sequence for $\M_X$).
			
	\begin{equation*}\includegraphics{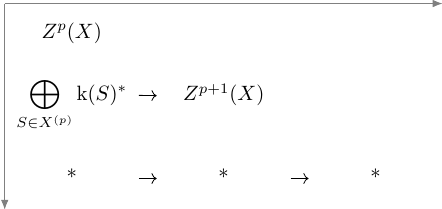}\end{equation*}
	The first two diagonal terms of this spectral sequence stabilize at $E_2$, and we have $E_\infty^{0,0}=Z^p(X)$, $E_\infty^{1,-1}=CH^{p+1}(X)$.
	The spectral sequence then yields a map \[c_0^p: K_0(\M_X^p) \to Z^p(X),\quad [\F] \mapsto \sum_{Z \in X^{(p)}} \operatorname{ln}_{\O_{X, Z}}(\F_Z) \cdot Z \in Z^p(X), \] which sends the class of a sheaf $\mathcal{F}$ to the proper-dimensional part of its support. 
	Here $X^{(p)}$ is the set of irreducible closed subvarieties of codimension $p$ in $X$, and $\operatorname{ln}_{\O_{X, Z}}(\F_Z)$ is the length of the module $\F_Z$ over the local ring $\O_{X,Z}$.
			
	There is also a map \[\sigma: Z^p(X) \to K_0(\M^p_X)\] given by $Z \mapsto [\O_Z]$, where $Z$ is an irreducible closed subvariety of codimension $p$ in $X$. This map provides a section of $c_0^p$, that is: \begin{equation*}c_0^p \circ \sigma = \operatorname{Id}: Z^p(X) \to Z^p(X). \end{equation*}
	Denote by $\operatorname{pr}$ the corresponding projection onto $\operatorname{Ker}(c_0^p)$, namely, \[{\operatorname{pr}: K_0(\M_X^p) \to \operatorname{Ker}(c_0^p)},\quad \operatorname{pr}=\operatorname{Id}-\sigma \circ c_0^p.\]
	From the spectral sequence one gets a homomorphism \[{\tau: \operatorname{Ker}(c_0^p) \to CH^{p+1}(X)}.\] 
	\begin{remk}\label{remarkstar} 
		The group $\operatorname{Ker}(c_0^p)$ coincides with the image of the map \\$\rho: K_0(\M_X^{p+1}) \to K_0(\M_X^p)$. For any $r \in K_0(\M_X^{p+1})$, the equality 
		\begin{equation*} \tau(\rho(r)) = \left[ c_0^{p+1}(r) \right]\end{equation*}
		holds in $CH^{p+1}(X)$.
	\end{remk}
	\begin{defn} We call the composition 
		\begin{equation*} 
			c_1^p = \operatorname{\tau} \circ \operatorname{pr}: K_0(\M_X^p) \to CH^{p+1}(X)
		\end{equation*}
		a \emph{generalized first Chern class}.
	\end{defn}
	For brevity, given a coherent sheaf $\F$, we denote $c_0^p(\left[\F\right])$ and $c_1^p(\left[\F\right])$, where $\left[\F\right]$ is the class of $\F$ in $K_0$, by just $c_0^p(\F)$ and $c_1^p(\F)$, respectively.
			
	Let us give a more explicit formula for $c_1^p$ in terms of the first Chern class on smooth varieties; it is not used in the proof of the main result, but we include it, because it is interesting in its own right.
	
	Suppose $\G$ is a coherent sheaf on $Z$, where $Z$ is an irreducible closed subvariety of $X$ of codimension $p$.
	Denote by $j:Z\hookrightarrow X$ the corresponding closed immersion.
	Consider the normalization map $\pi: \widetilde{Z} \to Z$, with $\widetilde{Z}_{sing}$ and $\widetilde{Z}_{sm}$ denoting the singular and smooth loci of $\widetilde{Z}$.
	The morphism $f\defeq j\circ\pi$ is finite and birational onto its image.
	As $\widetilde{Z}$ is normal, $\operatorname{codim}_{\widetilde{Z}}(\widetilde{Z}_{sing})\geq 2$, so that the groups $CH^i(\widetilde{Z})$ and $CH^i(\widetilde{Z}_{sm})$ are canonically identified for $i=0,1$.
	By abuse of notation, we will write \[f_*: CH^i(\widetilde{Z}_{sm}) \to CH^{p+i}(X)\] for the composition \[CH^i(\widetilde{Z}_{sm}) \xrightarrow{\sim} CH^i(\widetilde{Z}) \to CH^{p+i}(X),\] where $i=0,1$.
		
	Let $\eta: \G \to \pi_*\pi^*\G$ be the canonical morphism, and put \[\varphi = j_*\eta: j_*\G \to f_*\pi^*\G.\]
	\begin{prop}\label{chern-explicit}	In $CH^{p+1}(X)$, there is an identity
		\begin{equation*}c_1^p(j_* \G)=
			f_*\mathop{c_1} \left( \pi^*\G |_{\widetilde{Z}_{sm}} \right) +
			\left[c_0^{p+1}(\operatorname{Ker}{\varphi})\right] -
			\left[c_0^{p+1}(\operatorname{Coker}{\varphi})\right],
		\end{equation*}
		where $c_1\left( \pi^*\G |_{\widetilde{Z}_{sm}} \right) \in CH^1(\widetilde{Z}_{sm})$.
	\end{prop}
	\begin{proof}
		Consider the exact sequence
			\[
				0 \to 
				\operatorname{Ker}(\varphi) \to 
				j_*\G \xrightarrow{\varphi} f_*\pi^*\G \to \operatorname{Coker}(\varphi) \to 
				0\:.
			\]
				Since $c_1^p$ is a group homomorphism, we have:
				\[
					c_1^p(j_*\G)=
						c_1^p(f_*\pi^*\G)+
						c_1^p(\operatorname{Ker} \varphi)-
						c_1^p(\operatorname{Coker} \varphi). 
				\]
				It follows from Remark \ref{remarkstar} that 
				\[
					c_1^p(\operatorname{Ker} \varphi) =
						\left[c_0^{p+1}( \operatorname{Ker} \varphi )\right], \quad
					c_1^p(\operatorname{Coker}\varphi) =
						\left[c_0^{p+1}( \operatorname{Coker} \varphi )\right].
				\]
				
				As $f$ is a finite morphism, it induces exact functors $f_*: \M_{\widetilde{Z}}^i \to \M_X^{i+p}$ for $i \geq 0$.
				Hence, the Quillen spectral sequence is functorial with respect to $f_*$, and $c_1^p(f_*\H)=f_*c_1^0(\H)$ for any $\H \in \M_{\widetilde{Z}}$.
				Similarly, the Quillen spectral sequence is functorial with respect to the restriction from $\widetilde{Z}$ to $\widetilde{Z}_{sm}$, so $c_1^0(\H)|_{\widetilde{Z}_{sm}} = c_1^0\left(\H|_{\widetilde{Z}_{sm}}\right)$.
				Trivially, $c_1^0$ coincides with $c_1$ on smooth varieties. 
				Altogether, we have obtained the equality 
				\[
					c_1^p(f_*\H)=f_*c_1\left(\H|_{\widetilde{Z}_{sm}}\right);
				\]
				apply this to $\H=\pi^*\G$ to conclude.
		\end{proof}
		The Riemann--Roch theorem without denominators allows us to compare the map $c_1^p$ with the usual Chern classes.
		Let $\iota^p:\M_X^p \hookrightarrow \M_X$ be the standard embedding, and let $\iota^p_*: K_0(\M_X^p)\to K_0(\M_X)$ be the induced map.
		\begin{remk}\label{RR} Suppose $\operatorname{char} (\mainfield) = 0$ and $X$ is smooth.
		Then there is an identity 
		\[
			c_{p+1}\circ \iota^p_*={(-1)}^p p!\: c_1^p
		\]
		between maps from $K_0(\M_X^p)$ to $CH^{p+1}(X)$.
		
		The proof boils down to considering the following two cases:\\
		$1)$ Proving the identity for arguments in $K_0(\M_X^{p+1}) \subset K_0(\M_X^p)$.\\
		$2)$ Proving the identity for arguments of the form $\mathcal{F}=f_*\mathcal{G}$, where $\mathcal{G}$ is a coherent sheaf on a smooth variety $Z$ and $f: Z\to X$ is a finite morphism birational onto its image.
		
		In the first case, the statement follows from Remark \ref{remarkstar} and the Riemann--Roch theorem without denominators (see \cite{Fulton}*{Example $15.3.6$} for the specific statement needed).
		
		To deal with the second case one can apply the Riemann--Roch theorem without denominators for projective morphisms (see \cite{Pappas}*{Theorem$\: 2.2$}) to the morphism $f$.
		\end{remk}
\end{section}
\begin{section}{Convenient model for algebraic $K$-theory space}\label{model}
	Consider the category $\mathcal{E}$ of essentially small exact categories.
	There is a functor 
	\[
		{\mathbb{K}: \mathcal{E} \to \operatorname{Ho}(CW)}
	\]
	that assigns to an exact category $\mathcal{C}$ the homotopy type of its $K$-theory space (by which we understand the loop space of the space $BQ\C$ defined in \cite{Quillen}).
	We have $K_i(\mathcal{C})=\pi_i(\mathbb{K}(\mathcal{C}))$. 
	In particular, for any object $M$ of $\mathcal{C}$ there is a connected component \[ 
		[M] \in \pi_0(\mathbb{K}(\mathcal{C}))=K_0(\mathcal{C})
	\]
	corresponding to it, such that $[M\oplus N]=[M]+[N]$.
	To a pair of an object $A$ of $\mathcal{C}$ and an automorphism $\lambda: A \to A$, one can assign a canonical homotopy class of loops 
	\[
		[\gamma_\lambda]\in\pi_1(\mathbb{K}(\mathcal{C}))=K_1(\mathcal{C})
	\]
	in a way consistent with composition of automorphisms. 
	Notice that $\pi_0(\mathbb{K}(\mathcal{C}))$ and $\pi_1(\mathbb{K}(\mathcal{C}))$ are abelian groups, and so the homotopy groups are well-defined up to a canonical isomorphism, independently of the choice of a base point.
	We also note that both correspondences are natural with respect to exact functors $\mathcal{C}\to\mathcal{C}'$. 
	
	In what follows, we use a lift of the functor $\mathbb{K}$ to a functor 
	\[ 
		\mathbf{K}: \mathcal{E} \to CW 
	\]
	assigning to an exact category $\mathcal{C}$ an $H$-space $\mathbf{K}(\mathcal{C})$ equipped with the following data:
	
	$1)$ for any object $M$ of $\mathcal{C}$, the choice of a point $\kpt(M)\in \mathbf{K}(\mathcal{C})$ which belongs to the component $[M] \in \pi_0(\mathbb{K}(\mathcal{C}))$;
	
	$2)$ for any short exact sequence
	$E =
	\left[\: 0
	\rightarrow A
	\rightarrow	B
	\rightarrow	C
	\rightarrow 0 \: \right]$
	in $\mathcal{C}$, the choice of a path ${\gamma_E : \kpt(B) \dashrightarrow \kpt({A\oplus C})}$. Here and further we use a dashed arrow to keep track of the source and target point of a path.
	
	We require this data to satisfy the following conditions:
	
	{\it i}) for all $A, B$ in $\mathcal{C}$, there is an equality $\kpt(A)+\kpt(B)=\kpt({A\oplus B})$ as points of $\mathbf{K}(\mathcal{C})$;
	
	{\it ii}) for any automorphism $\lambda: A \to A$ in $\mathcal{C}$, the class $[\gamma_\lambda]$ equals $[\gamma_E]$, where ${
	E =
	[\: 0
	\rightarrow A
	\xrightarrow{\lambda}	A
	\rightarrow	0
	\rightarrow 0 \: ]
	}$;
	
	{\it iii}) for an exact functor $F:\mathcal{C} \to \mathcal{C}'$, we ask that $\kpt({F(A)})=F_*\kpt(A)$, and that $\gamma_{F(E)}=F_*\gamma_E$.
	
	An example of such construction is given by the simplicial set of Gillet--Grayson \cite{GilletGrayson}*{Theorem 3.1}.
	\end{section}
\begin{section} {Proof of the main result}	
	First, we find an expression for the differential
	\[
		d_2^p : E_2^{p,-p-1} \to E_2^{p+2,-p-2}
	\]
	in terms of a certain boundary map. The relevant part of the Quillen spectral sequence has the following form:
	\[\includegraphics{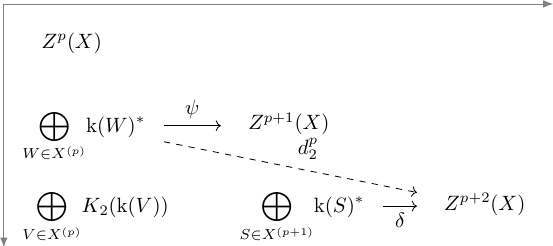}\]
	
	Note that the complement to the left column is precisely the spectral sequence sheet from Section \ref{Chern} with $p$ replaced by $p+1$. As explained therein, that truncated spectral sequence induces the maps
	\[
	c_0^{p+1}:K_0(\M_X^{p+1})\to Z^{p+1}(X),
	\]
	\[
	\tau:{\operatorname{Ker}}(c_0^{p+1})\to CH^{p+2}(X).
	\]
	In the long homotopy exact sequence associated with the homotopy fibration
	\[
	\K(\M_X^{p+1})\to \K(\M_X^{p})\to \K(\M_X^p/\M_X^{p+1}),
	\]
	we have a boundary map
	\[
	\partial: K_1(\M_X^p/\M_X^{p+1})\cong \displaystyle{\bigoplus_{W\in X^{(p)}}}\mainfield(W)^* \to K_0(\M_X^{p+1}).
	\]
	Note that $\psi=c_0^{p+1}\circ \partial$.
	
	Consider an element $\alpha\in {\operatorname{Ker}}(\psi)$. Then $\partial(\alpha)\in {\operatorname{Ker}}(c_0^{p+1})$. It follows from the general description of the differential in the spectral sequence of an exact couple that there is an equality
	\[
	d_2^p(\alpha)=\tau(\partial(\alpha))
	\]
	in ${\operatorname{Coker}}(\delta)\cong CH^{p+2}(X)$.
	In the notation introduced in Section \ref{statement}, this amounts to the following statement.
	
	\begin{lemma}\label{lem1}
		Let $\alpha=\{(W_i,\varphi_i)\}\in {\operatorname{Ker}}(\psi)$. Then $d_2^p(\alpha)=\sum_i c_1^p(\operatorname{\partial}(W_i,\varphi_i))$.
	\end{lemma}
	
	Now let us evaluate the boundary map $\partial$. Let $W$ be an irreducible codimension $p$ cycle on $X$, and let $\varphi$ be a non-zero rational function on $W$. Specializing to the choice $(W_i, \varphi_i)=(W,\varphi)$, we will write $j:W\hookrightarrow X$, $\L$, $\F$, $\G$, $f$, $g$ for the objects $j_i:W_i\hookrightarrow X$, $\L_i$, $\F_i$, $\G_i$, $f_i$, $g_i$ defined in Section \ref{statement}.
	
	\begin{lemma}\label{lem2}
		There is an equality $\operatorname{\partial}(W,\varphi)=[\F]-[\G]$ in $K_0(\M_X^{p+1})$.
	\end{lemma}
	\begin{proof}
		By the properties of the functor $\KK$ specified in Section \ref{model}, the exact sequences
		\[
			0\to j_*\O_W\xrightarrow{f} j_*\L\to \F\to 0,
		\]
		\[
			0\to j_*\O_W\xrightarrow{g} j_*\L\to \G\to 0
		\]
		define corresponding paths in $\KK(\M^p_X)$:
		\[
			\gamma_1:{\kpt}(j_*\L)\dasharrow {\kpt}(j_*\O_W)+{\kpt}(\F),
		\]
		\[
			\gamma_2:{\kpt}(j_*\L)\dasharrow {\kpt}(j_*\O_W)+{\kpt}(\G).
		\]
		Consider the composition $\gamma_1\circ\gamma_2^{-1}$. Since $\KK(\M_X^p)$ is an $H$-space, there is a path in $\KK(\M_X^p)$
		\[
			\gamma:{\kpt}(\G)\dasharrow {\kpt}(\F)
		\]
		such that its shift ${\kpt}(j_*\O_W)+\gamma$ is homotopy equivalent to $\gamma_1\circ\gamma_2^{-1}$. The path $\gamma$ defines a loop $\ell$ in $\KK(\M_X^p/\M_X^{p+1})$, which is homotopy equivalent to the loop in $\KK(\M_X^p/\M_X^{p+1})$ defined as the image of $\gamma_1\circ\gamma_2^{-1}$.
		
		Furthermore, by the property $(ii)$ of the functor $\KK$ (see Section \ref{model}), the paths $\gamma_1$ and $\gamma_2$ represent the elements
		\[
			f,g\in \mainfield(W)^*\subset K_1(\M_X^p/\M_X^{p+1}).
		\]
		This implies the equality $[\ell]=(W,\varphi)$ in $K_1(\M_X^p/\M_X^{p+1})$. Altogether, this proves the lemma.
	\end{proof}
	\noindent Now Theorem \ref{main} follows directly from Lemma \ref{lem1} and Lemma \ref{lem2}.
	
	\subsection*{Acknowledgments}
	I would like to thank my advisor Sergey Gorchinskiy for his constant support and guidance throughout this work. 
	I am grateful to Alexander Kuznetsov and Serge Yagunov for their careful reading of this text and useful discussions, and to the anonymous reviewer for valuable comments.
	Serge Yagunov also kindly explained to me his results on differentials in the motivic spectral sequence.
	
\end{section}


\begin{bibdiv}
\begin{biblist}
	\bib*{LNM341}{book}{
		booktitle={In Algebraic K-Theory I, Lecture Notes in Math., vol. 341},
		publisher={Springer--Verlag Berlin Heidelberg},
		date={1973}
	}
	\bib{Bloch}{article}{
		title={$K_2$ and algebraic cycles},
		author={Bloch, S.},
		journal={Ann. of Math.},
		volume={99 (2)},
		pages={349--379},
		date={1974}
	}
	\bib{BrownGersten}{collection.article}{
		author={Brown, K.},
		author={Gersten, S.},
		title={Algebraic K-theory as generalized sheaf cohomology},
		xref={LNM341},
		pages={266--292},
		year={1973}
	}
	\bib{Buchstaber}{article}{
		title={Modules of differentials of the Atiyah--Hirzebruch spectral sequence},
		author={Buchstaber, V. M.},
		date={1969},
		language={in Russian},
		journal={Mat. Sb. (N.S.)},
		volume={78 (120)},
		number={2},
		pages={307--320}
	}
	\bib{Fulton}{book}{
		author={Fulton, W.},
		title={Intersection theory, 2nd edition},
		publisher={Springer--Verlag New York	},
		date={1998}
	}
	\bib{GilletGrayson}{article}{
		author={Gillet, H.},
		author={Grayson, D.},
		title={The loop space of the Q-construction},
		journal={Ill. J. Math.},
		volume={11},
		date={1987},
		pages={574--597}
	}
	\bib{div0}{book}{
		author={Greuel, G.-M.},
		author={Lossen, C.},
		author={Shustin, E. I.},
		title={Singular Algebraic Curves},
		subtitle={With an Appendix by Oleg Viro},
		publisher={Springer International Publishing},
		year={2018}
	}
	\bib{Landsburg}{article}{
		author={Landsburg, S. E.},
		title={Relative Chow groups},
		journal={Illinois J. Math.},
		volume={35 (4)},
		pages={618-641},
		date={1991}
	}
	\bib*{DM18}{book}{
		title={In Quadratic forms, linear algebraic groups, and cohomology},
		publisher={Developments in Mathematics vol. 18, Springer New York}
	}
	\bib{Merkurjev}{article}{
		author={Merkurjev, A.},
		title={Adams operations and the Brown--Gersten--Quillen spectral sequence},
		xref={DM18},
		pages={305--313},
		date={2010}
	}
	\bib{Pappas}{article}{
		author={Pappas, G.},
		title={Integral Grothendieck--Riemann--Roch theorem},
		journal={Invent. math.},
		volume={170 (3)},
		date={2007},
		pages={455--481}
	}
	\bib{Quillen}{collection.article}{
		author={Quillen, D.},
		title={Higher algebraic K-theory: I},
		xref={LNM341},
		pages={85--147},
		year={1973}
	}
	\bib{Yagunov}{article}{
		title={Motivic cohomology spectral sequence and Steenrod operations},
		author={Yagunov, S.},
		date={2016},
		journal={Compositio Math.},
		volume={152},
		pages={2113--2133}
	}
\end{biblist}
\end{bibdiv}
\end{document}